\documentclass[12pt,leqno,oneside]{amsart}

\usepackage{amsmath,amsthm,amssymb,amsfonts,ifpdf}

\ifpdf
\usepackage[pdftex,pdfstartview=FitH,pdfborderstyle={/S/B/W 1},%
colorlinks=true, linkcolor=blue, urlcolor=blue, citecolor=blue,%
pagebackref=true]{hyperref}
\else
\usepackage[hypertex]{hyperref}
\fi

\usepackage[abbrev,msc-links,nobysame]{amsrefs}

\newtheorem {thm}   {Theorem}
\newtheorem {lem}      [thm]    {Lemma}

\newtheorem {prp}[thm]  {Proposition}

\newcounter{AbcT}

\renewcommand{\a}{\alpha}

\newcommand{\e}{\varepsilon}

\newcommand{\g}{\gamma}

\newcommand{\R}{{\bf R}}

\newcommand{\Z}{{\bf Z}}
\newcommand{\C}{{\bf C}}

\newcommand {\cP} {{\mathcal P}}

\newcommand {\cQ} {{\mathcal Q}}

\newcommand{\be}{\begin{equation}}
\newcommand{\ee}{\end{equation}}
\newcommand{\bea}{\begin{eqnarray}}
\newcommand{\eea}{\end{eqnarray}}
\newcommand{\bean}{\begin{eqnarray*}}
\newcommand{\eean}{\end{eqnarray*}}

\newcommand{\mat}[4]
{\left(
\begin{array}{cc}
#1 & #2 \\
#3 & #4
\end{array}
\right)}

\newcommand{\nri}[1]{\lfloor #1 \rceil}
\newcommand{\igap}{\:}

\DeclareMathOperator{\SL}{SL}

\DeclareMathOperator{\SO}{SO}
\DeclareMathOperator{\SU}{SU}

\title[Diophantine property]%
{Diophantine property in the group of affine transformations of the line}
\author{P\'eter P. Varj\'u}
\thanks{
I gratefully acknowledge the support
of the Simons Foundation and the European Research Council
(Advanced Research Grant 267259)%
}

\date{\today}

\keywords{solvable Lie groups, Diophantine property, roots of polynomials, zeros of polynomials}

\subjclass[2010]{22E25, 30C15, 11C08}

\begin{document}

\begin{abstract}
We investigate the Diophantine property of a pair of elements in the
group of affine transformations of the line.
We say that a pair of elements $\g_1,\g_2$ in this group is Diophantine if there is a 
number $A$ such that a
product of length $l$ of elements of the set $\{\g_1,\g_2,\g_1^{-1},\g_2^{-1}\}$ is
either the unit element or of distance at least $A^{-l}$ from the unit element.
We prove that the set of non-Diophantine pairs in a certain one parameter family
is of Hausdorff dimension 0.
\end{abstract}

\maketitle

\section{Introduction}
\label{section:introduction}

Let $G$ be a group endowed with a distance $d(\cdot,\cdot)$.
At this point, we make no assumption on $d(\cdot,\cdot)$, but later we
will specialize to concrete groups which come with natural metrics.
We denote by 1 the unit element of $G$ (as well as the unit element of any other multiplicative group).

Let $\g_1,\g_2\in G$ be two arbitrary elements and write $S=\{\g_1,\g_2,\g_1^{-1},\g_2^{-1}\}$.
We denote the set of words of length at most $l$ over $S$ by
\[
W_l:=\{g\in G:{\;\rm there\: is\;} s_1,\ldots, s_l\in S {\; \rm such\: that\;} g=s_1\cdots s_l\}.
\]

In this paper we study the behavior of 
\[
d_l:=\min_{g\in W_l:g\neq1}d(g,1)
\]
as $l$ increases.
This, off course, depends on the properties of $G$ and the generators.
Following Gamburd, Jakobson and Sarnak \cite{GJS-SU2}, we say that the pair 
$\g_1,\g_2$ is Diophantine,
if there is a number $\beta>0$ independent of $l$ such that
\[
d_l\ge|W_l|^{-\beta}.
\] 

Suppose that $G$ is isomorphic and isometric  to $(\R,+)$ and $\g_1,\g_2$ is mapped to the real numbers $x,y$
by this isomorphism.
Then the behavior of $d_l$ depends on the distance of $x/y$ to rational numbers of a given
denominator.
This is a classical topic and much is known about it.
It was proved by Khintchine \cite{Khintchine-diophantine}
that Lebesgue almost every generators, $d_l\ge Cl^{-1-\e}$ for every $\e$
with some constant depending on $\e$ and $x,y$.
We will outline a proof for this fact  later.
In addition, when $x/y$ is an algebraic number then the same inequality holds.  
This latter claim is equivalent to a deep result of Roth \cite{Roth-diophantine}.
Thus, almost every pairs
$\g_1,\g_2\in G\cong(\R,+)$
are Diophantine and so are those that are mapped to algebraic numbers.

On the other hand, for a generic pair $\g_1,\g_2$ in the sense of Baire category,
we have $d_l\le f(l)$ infinitely often with an arbitrarily fast decreasing  positive function $f$.
Therefore, generic pairs in the sense of Baire category are not Diophantine.

Now we turn to non-commutative groups.
Gamburd, Jakobson and Sarnak \cite{GJS-SU2} noticed that every algebraic  $\g_1,\g_2\in G=\SU(2)$
are Diophantine, and also that this is not the case for generic pairs in the sense of Baire category.
They posed the problem (see \cite[Problem (4) at the end of Section 1]{GJS-SU2})
whether almost every pair  possesses the Diophantine property.
This problem is of great interest due to a result of Bourgain and Gamburd \cite{BG-SU2}.
This asserts that the random walk on $\SU(2)$ generated by $\g_1,\g_2$ is mixing exponentially fast,
if the pair $\g_1,\g_2$ is Diophantine.
See also their paper \cite{BG-SUd-arxiv}, where the result is extended to $\SU(d)$.

Kaloshin and Rodnianski \cite[Theorem 1]{KR-diophantine}
proved that for almost every pair  $\g_1,\g_2\in G=\SO(3)$, there is a number
$D>1$ such that $d_l\ge D^{-l^2}$.
This bound is weaker than what is required in the definition of Diophantine property.

The problem in nilpotent Lie groups is studied by Aka, Breuillard,  Rosenzweig and Saxc\'e
\cite{ABRS-Diophantine-Nilpotent}.
They prove that almost every tuple is Diophantine in a nilpotent group defined over the rationals
\cite[Corollary 1.2]{ABRS-Diophantine-Nilpotent}.
Moreover, for nilpotent groups of step at most 5, they prove this without the rationality assumption.
On the other hand, they exhibit a nilpotent Lie group of step 6 in which almost every $k$-tuple is
not Diophantine for $k\ge3$ \cite[Theorem 1.3]{ABRS-Diophantine-Nilpotent}.
For almost every pair of generators, the Diophantine property holds in groups of step 6, as well,
and a counterexample exists among the groups of step 7.

The purpose of the paper is to study the problem of Diophantine property in a solvable Lie group.
More specifically, we consider the group:
\be\label{equation:G}
G=\left\{\mat ab01:a,b\in\C\right\}.
\ee
Observe that the line $\{(x,1)^t:x\in \C\}$ is invariant under the action of $G$, and
\[
\mat ab01
\left(\begin{array}{c} x \\ 1 \end{array} \right)
= \left(\begin{array}{c} ax+b \\ 1 \end{array} \right),
\]
hence $G$ is isomorphic to the group of affine transformations of the (complex) line.

We further restrict generality and consider pairs only of the form
\be\label{equation:generators}
\g_1=\mat x001,\quad \g_2=\mat 1101.
\ee
This is indeed a significant loss of generality, but we note that an arbitrary pair
can be conjugated into this form, if one of them is unipotent, i.e
all of its eigenvalues are 1.

The result of the paper is the following:
\begin{thm}\label{theorem:main}
The set of complex numbers $|x|>1$ such that the pair \eqref{equation:generators}
is not Diophantine is of Hausdorff dimension 0.
\end{thm}

Since \eqref{equation:generators} generates the same group if we replace $x$ by $x^{-1}$,
the theorem implies the same conclusion for parameters  $|x|<1$ .
On the other hand we do not claim anything for parameters on the unit circle.

\subsection*{Acknowledgment}
I am greatly indebted to Menny Aka and Emmanuel Breuillard for communicating the problem to me.
I thank them and Elon Lindenstrauss, Lior Rosenzweig, and Nicolas de Saxc\'e for stimulating
discussions about this project.

\section{Outline}
\label{section:outline}

Before discussing the proof of Theorem \ref{theorem:main}, we briefly review
the Abelian case.
This argument is folklore.
Assume that $G$ is isomorphic and isometric to $(\R,+)$ and the pair $\g_1,\g_2$ is mapped into $x,y$.
In fact, we can choose the isomorphism in such a way that $|x|<y=1$.
This changes the metric on $G$, but only by a constant factor.

Observe that any element of $W_l$ is of the form $mx+n$, where $m,n\in\Z$ and $|m|+|n|\le l$.
Fix a number $A>0$.
If $|mx+n|\le l^{-A}$, then $|x+n/m|\le m^{-1}l^{-A}\le m^{-(1+A)}$.
Thus if $d_l\le l^{-A}$ for infinitely many $l$, then $x$ belongs to the set
\[
\Omega_A:=\bigcap_{k=1}^\infty\bigcup_{m=k}^{\infty}\bigcup_{n=-m}^{m}
\left[\frac{n}{m}-m^{-(1+A)},\frac{n}{m}+m^{-(1+A)}\right],
\]
where $[a,b]\subset\R$ denotes the closed interval between $a$ and $b$.

Let $\a>0$ be a number.
Then the $\a$-dimensional Hausdorff measure of $\Omega_A$ is at most
\[
\lim_{k\to\infty}\sum_{m=k}^{\infty}(2m+1)\left(2m^{-(1+A)}\right)^\a=0,
\]
if $\a>2/(1+A)$.
Thus the Hausdorff dimension of $\Omega_A$ is at most $2/(1+A)$.
This shows that the set of pairs of generators $\g_1,\g_2\in G\cong(\R,+)$
that are not Diophantine is of Hausdorff dimension 0.

Our proof of Theorem \ref{theorem:main} follows similar lines.
First, we need to understand how elements of $W_l$ look like.
This is achieved in the next lemma.

\begin{lem}\label{lemma:description}
Let $G$ be as in \eqref{equation:G} and $\g_1,\g_2\in G$ of the form \eqref{equation:generators}.
Then any element of $W_l$ is of the form
\be\label{equation:Wl}
\mat{x^k}{x^{-n}P(x)}{0}{1},
\ee
where $-l\le k\le l$ and $0\le n \le l$ are integers and $P(x)=a_0+a_1x+\ldots +a_mx^m$ is a polynomial
with integer coefficients of degree at most $l+n$ and $\sum_{i=0}^m|a_i|\le l$. 
\end{lem}

Although elements of form \eqref{equation:Wl} are not all contained in $W_l$, they are easily seen to be
contained in $W_{10l}$.
This implies, in particular, that $|W_l|$ grows exponentially.

\begin{proof}
The proof is by induction.
The statement is easy for $l=0$.
Multiplying an element of the form \eqref{equation:Wl} by a generator from the left, we either
increase or decrease the value of $k$ by 1, or add $\pm x^k$ to $x^{-n}P(x)$.
\end{proof}

If an element of $W_l$ differs from the unit element in the top left entry, then its distance
{}from 1 is at least
$1-|x|^{-1}$ (recall that $|x|>1$), which is a bound independent of $l$.
Hence it is enough to consider elements with $k=0$ and hence $P(x)\not\equiv 0$ in \eqref{equation:Wl}.
To implement an argument similar to the Abelian case, we need to understand the set, where a polynomial
satisfying the properties in Lemma \ref{lemma:description} is small.
This was off course simpler in the Abelian case, since there we only needed to consider
polynomials of degree 1.

In a semi-simple group like $\SU(2)$ or $\SL_2(\R)$, the elements of $W_l$ can be described in a similar fashion,
but the polynomials which occur can have exponentially large coefficients.
Then the problem becomes even more difficult.
This was considered by Kaloshin and Rodnianski \cite{KR-diophantine}, and they gave
a bound  of the form $d_l\ge A^{-l^2}$.
The reason why we are able to give a stronger bound in the case we consider is that we have
a better estimate on the coefficients and this allows us to control the distance between the roots.

We write
\[
\cP_l:=\{P(x)=\sum_{i=0}^{2l}a_ix^i:a_i\in\Z {\rm\;and\;} \sum_{i=0}^{2l}|a_i|\le l\}.
\]
This is the set of polynomials that can occur in \eqref{equation:Wl}.
Our aim thus is to understand the set, where an element of $\cP_l$ is small.
This set is easily understood in terms of the roots of the polynomial.
Indeed, by the fundamental theorem of algebra, we can write
\[
P(x)=a_{m}\prod_{i=1}^{m}(x-z_i),
\]
where $a_{m}$ is the highest order non-zero coefficient and $z_1,\ldots,z_{m}$ are the roots of $P(x)$.

We see from this formula, that $P(x)$ is small only in the proximity of the roots.
If the roots are "far from each other", then $P(x)<A^{-l}$ for a suitably large number $A$
will hold only on $l$ very small regions around the roots.
The diameter of the regions is approximately $(A')^{-l}$, where $A'$ is another number depending on
what "far from each other" means.
On the other hand, if there are $k$ roots "very close to each other", then $P(x)<A^{-l}$
may hold on a region of diameter $(A')^{-l/k}$.

Thus we need to show that a typical polynomial does not have many roots near each other, and
we need to estimate the number of exceptions.
We are able to do this only near points which are bounded away from the unit circle.
We choose a parameter $r>0$ and prove that $P$ can have at most $C\log l$ roots
"near" to a point $x\in\C$ satisfying $|x|>1+r$ and if it has "many" such roots, then it has
a "large" coefficient.
($C$ is a number depending on $r$.)
In fact, we prove much more, because we only use that the roots close to $x$ are outside the
circle of radius $1+(1/2)r$ around the origin.
The precise statement will be given later in Lemma \ref{lemma:Jensen}, which is a simple application
of Jensen's formula.

In Section \ref{section:decomposition}
we decompose the annulus $\{x:1+r\le|x|\le r^{-1}\}$ into many small regions, and we consider
those elements of $\cP_l$ which have many roots close to each other within a single region.
As it turns out, the difference of two such polynomials also have many roots within the same region.
Then by  Lemma \ref{lemma:Jensen}, we can conclude that the coefficients of the two polynomials must
be far away in the $l^\infty$ norm.
Finally, in Section \ref{section:l1linf} we estimate the number of $l^\infty$-separated points
in an $l^1$-ball of $\Z^{2l+1}$.
Owing to the definition, the coefficients of a polynomial in $\cP_l$
are bounded in $l^1$ norm, hence we find an estimate on the
number of polynomials which have many roots close to each other in the annulus  $\{x:1+r\le|x|\le r^{-1}\}$.

The details of this argument is given in Section \ref{section:details}, where we prove the
following proposition.

\begin{prp}\label{proposition}
For any numbers $r>0$, $a>1$ there are numbers $A,C>1$ such that the following holds.
For an integer $k\ge1$,
denote by $\cQ_{l,k}\subset\cP_l$ the set of polynomials $P(x)$ for which the
set
\[
\Omega_{P,l}:=\{x:|P(x)|<A^{-l}{\;\rm and\; }1+r\le|x|\le r^{-1}\}
\]
can not be covered by $2l$ disks of radii $2^{-a l/k}$.
Then for every integer $k\ge1$, we have $|\cQ_{l,k}|<C\cdot10^{l/k}$.
Moreover, for $k>\log l$, $\cQ_{l,k}$ contains only the 0 polynomial.
\end{prp}

It is easy to deduce Theorem \ref{theorem:main} from Proposition \ref{proposition} and
Lemma \ref{lemma:description}.
Indeed, by Lemma \ref{lemma:description}, it follows that if $\g_1,\g_2$ given by
\eqref{equation:generators} is not Diophantine, then
\[
x\in\Omega:=\bigcap_{n=1}^{\infty}\bigcup_{l=n}^\infty\bigcup_{P\in\cP_l}\Omega_{P,l},
\]
provided $r$ is sufficiently small, so that $1+r\le|x|\le r^{-1}$.

If $P\notin\cQ_{k,l}$, then we can cover $\Omega_{P,l}$ by $2l$ disks of radii $2^{-a l/k}$.
Thus we see that the $\a$ dimensional Hausdorff measure
of $\Omega$ is at most
\be\label{equation:Hmeasure}
\lim_{n\to\infty}\sum_{l=n}^{\infty}\left(|\cP_l|\cdot2l\cdot2^{-\a a l}
+\sum_{k=1}^{[\log l]}|\cQ_{l,k}|\cdot 2l\cdot 2^{-\a a l/(k+1)}\right).
\ee

We have 
\[
|\cP_l|\le 2^{2l+1}\cdot\binom{3l}{2l}\le100^l.
\]
Here, the power of 2 stands for the possible choices of the sign and the binomial coefficient
is the number of ways we can write $l$ as the sum of $2l+1$ non-negative integers taking into account
the order of the terms.
In addition, by Proposition \ref{proposition}, we have
\[
|\cQ_{l,k}|\le C\cdot10^{l/k}\le C\cdot100^{l/(k+1)}.
\]

Now \eqref{equation:Hmeasure} is bounded above by
\[
\lim_{n\to\infty}\sum_{l=n}^{\infty}\sum_{k=0}^{[\log l]}2Cl\cdot 100^{l/(k+1)}2^{-\a a l/(k+1)}.
\]
This limit is easily seen converging to 0, if we choose $a$ sufficiently large so that
\[
2^{\a a}>100.
\]
This shows that  the set of numbers $x$ within the annulus
$\{x\in\C:1+r\le|x|\le r^{-1}\}$ such that the pair $\g_1,\g_2$ given by
\eqref{equation:generators} is not Diophantine has Hausdorff dimension at most $\a$.
We can take $r\to0$ and $\a\to 0$.
Therefore Theorem \ref{theorem:main} will be proved once we verified Proposition \ref{proposition}.

\section{Bounding the number of exceptional polynomials}
\label{section:details}

\subsection{}
\label{section:Jensen}
We begin by giving the Lemma, which bounds the number of large roots of a polynomial in terms
of the maximum of the coefficients.
Note that the bound does not depend on the degree of the polynomial.
\begin{lem}\label{lemma:Jensen}
For every number $r>0$, there is a number $C_r$ such that the following is true.
If
\[
P(x)=a_mx^m+\ldots a_1x+a_0=a_m\prod_{i=1}^{m}(x-z_i)
\]
is a polynomial with integer coefficients, then
\[
|\{i:|z_i|>1+r/2\}|\le C_r[\log(\max_{0\le i\le m}|a_i|)+1].
\]
\end{lem}

\begin{proof}
We apply a version of Jensen's formula on the circle $|x|=\rho=(1+r/2)^{1/2}$.
We obtain
\bean
\log\left[\sum_{i=0}^{m}|a_i|\rho^i\right]
&\ge&\frac{1}{2\pi\rho}\int_{|x|=\rho}\log|P(x)|dx\\
&=&\frac{1}{2\pi\rho}
\int_{|x|=\rho}\left(\log|a_m|+\sum_{i=1}^m\log|x-z_i|\right)dx\\
&=&\log|a_m|+\sum_{i=1}^m\frac{1}{2\pi\rho}\int_{|x|=\rho}\log|x-z_i|dx\\
&=&\log|a_m|+\sum_{i:|z_i|>\rho}\log|z_i|+\sum_{i:|z_i|\le\rho}\log\rho.
\eean
Taking exponentials and dividing both sides by $\rho^m$, we get
\[
\sum_{i=0}^m|a_i|\rho^{i-m}\ge\prod_{i:|z_i|>\rho}\frac{|z_i|}{\rho}\ge\rho^{|\{i:|z_i|>1+r/2\}|}.
\]
(Recall that $\rho^2=1+r/2$.)

Finally, we observe that
\[
\sum_{i=0}^m|a_i|\rho^{i-m}\le \frac{\rho}{\rho-1}\max_{0\le i\le m}|a_i|.
\]
Combining this with the previous inequality and taking logarithms we deduce the claim.
\end{proof}

\subsection*{Note added in proof}
I am grateful to Boris Solomyak, who
pointed out to me that a version of the above Lemma \ref{lemma:Jensen} and a similar proof
can be found in the paper \cite{BBBP-multiple-roots}.
See in particular Theorem 2 and its refinement in Section 4
in that paper.

\subsection{}
\label{section:decomposition}
We fix some real numbers $r>0$, $a>1$, $A$ and $B$ whose value will be determined later.
We write the annulus as a union of small regions:
\be\label{equation:decomposition}
X:=\{x\in\C:1+r\le|x|\le r^{-1}\}=X_1\cup X_2\cup\ldots\cup X_N
\ee
in such a way that the diameter of each $X_i$ is at most $2^{-l/k}$, each of them
contains a disk of radius $c2^{-l/k}$ and the number
of the regions is $N\le C4^{l/k}$.
Here and everywhere below, $C$ and $c$ denote numbers which may depend only on $r$ and $a$
and their values may change from line to line.

We write
\[
\cQ_{l,k,i}:=\{P\in\cP_l:|P(x)|\le B^{-l} {\;\rm for\: all\;} x\in X_i\}.
\]
In the next lemma we exploit the idea that an element of $\cQ_{l,k}$ must have approximately
at least $k$ zeros near each other within the annulus $X$.
Hence it must be fairly small on one of the regions $X_i$.
Note however that the $X_i$ are significantly larger than the disks of radius
$2^{-a l/k}$ that fail to cover the set where the polynomial is below $A^{-l}$.
Thus we have to choose $B$ significantly smaller than $A$ in order that the lemma
holds.

\begin{lem}\label{lemma:decomposition}
Suppose that $B\le c\cdot A^{1/a}$
for some suitably small number $c$ depending on $r$.
Then
\[
\cQ_{l,k}\subset \cQ_{l,k,1}\cup\ldots\cup \cQ_{l,k,N}.
\]
\end{lem}
\begin{proof}
Let $P\in\cQ_{l,k}$ and let
\[
P(x)=a_{m}\prod_{i=1}^{m}(x-z_i).
\]

By the definition of $\cQ_{l,k}$, the disks of radii $2^{-a l/k}$ around the roots do not cover the set
where $P$ is below $A^{-l}$.
Thus there is a point $x_0\in X$ such that $|P(x_0)|<A^{-l}$ and $|x_0-z_i|>2^{-a l/k}$
for all $1\le i\le m$.

There is a region $X_j$ such that $x_0\in X_j$.
We show that $P\in\cQ_{l,k,j}$, that is $P(x)\le B^{-l}$ for all $x\in X_j$.
For each root $z_i$, we have
\[
|x-z_i|\le|x_0-z_i|+|x_0-x|\le|x_0-z_i|+2^{-l/k}.
\]
Since $|x_0-z_i|>2^{-a l/k}$,
\[
|x-z_i|\le|x_0-z_i|+|x_0-z_i|^{1/a}.
\]

If $|x_0-z_i|\le 1$, then $|x-z_i|\le2|x_0-z_i|^{1/a}$.
In the opposite case
\[
|x-z_i|
\le2|x_0-z_i|
\le2(|x_0|+|z_i|)
\le4r^{-1}\max\{1,|z_i|\}.
\]
Thus, in either case we have
\[
|x-z_i|
\le4r^{-1}\max\{1,|z_i|\}|x_0-z_i|^{1/a}.
\]

Therefore
\[
|P(x)|=|a_{m}|\prod_{i=1}^{m}|x-z_i|\le|a_{m}|\prod_{i=1}^{m}[4r^{-1}\max\{1,|z_i|\}|x_0-z_i|^{1/a}].
\]
By Jensen's formula,
\[
|a_m|\prod_{i=1}^{m}\max\{1,|z_i|\}=\exp\left[\frac{1}{2\pi}\int_{|x|=1}\log|P(x)|\igap dx\right]\le l.
\]
Thus
\[
|P(x)|\le l(4r^{-1})^m|P(x_0)|^{1/a}\le  \left(\frac{r^2}{16l^{1/l}}A^{1/a}\right)^{-l},
\]
since $m\le 2l$.
This gives $|P(x)|\le B^{-l}$ provided
\[
c\le \frac{r^2}{16l^{1/l}}.
\]
This can be bounded by a number depending only on $r$.
\end{proof}

\subsection{}
\label{section:l1linf}

In this section, we give an estimate on the size of $\cQ_{l,k,i}$.
Multiplying this by the number of the regions, we will conclude Proposition \ref{proposition}.
Taking the difference of two elements in $\cQ_{l,k,i}$, we get a polynomial which is
very small on the region $X_i$.
This requires that it has many zeros near this region and by Lemma \ref{lemma:Jensen},
it must have a large coefficient.
This is the content of the next lemma.

\begin{lem}\label{lemma:l1linf}
Suppose that $B$ is sufficiently large depending on $a$ and $r$.
Let $P\neq Q\in\cQ_{l,k,i}$.
Then the polynomial $P-Q$ has a coefficient greater than $e^{10k}$.
\end{lem}

\begin{proof}
Let
\[
P(x)-Q(x)=a_{m}(x-z_1)\cdots(x-z_{m}),
\]
and denote by $M$ the number of roots $z_i$ of modulus larger than $1+r/2$.
Recall that $X_i$ contains a disk of radius $c\cdot 2^{-l/k}$.
This implies that there is a point $x_0\in X_i$ such that $|x_0-z_i|\ge c\cdot 2^{-l/k}/\sqrt{M}$.
Moreover, if $|z_i|\le1+r/2$, then $|x_0-z_i|\ge r/2$.

Thus
\[
2B^{-l}\ge |P(x_0)-Q(x_0)|\ge (r/2)^{m}\left[c\frac{2^{-l/k}}{\sqrt{M}}\right]^M.
\]
By Lemma \ref{lemma:Jensen}, we see that $M\le\sqrt{l}$, otherwise $P-Q$
would have coefficients larger than $e^{\sqrt{l}/C_r-1}$, which is impossible if $l$
is sufficiently large.
Hence
\[
\frac{1}{2}(r/2)^{m}\left[\frac{c}{\sqrt M}\right]^M\ge C^{-l}
\]
for some constant $C$ depending only on $r$ and $a$.
Combining with the previous inequality, we get
${B}/{C}\le 2^{M/k}$, hence
\[
M\ge\frac{\log B-\log C}{\log 2}k.
\]

We combine this bound with Lemma \ref{lemma:Jensen} and conclude that
$P-Q$ must have a coefficient larger than
\[
e^{\frac{\log B-\log C}{C_r\log 2}k-1},
\]
which implies the claim for $B$ sufficiently large.
\end{proof}

\begin{proof}[Proof of Proposition \ref{proposition}]
First choose $B$ sufficiently large so that the conditions of Lemma \ref{lemma:l1linf} hold
and then choose $A$ sufficiently large so that  the conditions of Lemma \ref{lemma:decomposition}
hold, as well.

Applying Lemma \ref{lemma:l1linf} for $Q=0$ and any $0\neq P\in\cQ_{l,k,i}$, it follows that
$P$ must have a coefficient of size larger than $e^{10k}$.
This is only possible if $e^{10k}\le l$ that is $k\le (\log l)/10$.

We give an estimate on the size of $\cQ_{l,k,i}$.
Denote by $\nri{x}$ the nearest integer to the real number $x$.
(When $x$ is half plus an integer, we define $\nri{x}=x-1/2$.)
Write $K=e^{10k}$.
By Lemma \ref{lemma:l1linf}, the map
\[
\Phi:a_{2l}x^{2l}+\ldots+a_1x+a_0\mapsto
(\nri{a_{2l}/K},\ldots,\nri{a_{1}/K},\nri{a_{0}/K})
\]
is injective from $\cQ_{l,k,i}$ to $\Z^{2l+1}$.

For $v=(b_{2l},\ldots,b_0)\in\Z^{2l+1}$ we write
\[
\|v\|_1=|b_{2l}|+\ldots+|b_0|.
\]
Observe that $|\nri{x/K}|\le2|x|/K$ for all real numbers $x$.
Hence, for every $P\in\cQ_{l,k,i}\subset\cP_l$, we have
\[
\|\Phi(P)\|_1\le 2l/K.
\]

The number of ways of writing $[2l/K]$ as the sum of at most $2l+1$ non-negative
integers (taking the order into account) can be estimated using Stirling's formula, as follows:
\bean
\binom{[2l/K]+2l}{2l}
&\le&10\frac{(2l/K+2l)^{2l/K+2l+1/2}}{(2l)^{2l+1/2}(2l/K)^{2l/K+1/2}}\\
&\le&10\left(1+\frac{1}{K}\right)^{2l+1/2}\cdot(1+K)^{2l/K}\\
&\le&10 \exp[(2l+1/2)/K+2l\log(K+1)/K].
\eean
For each such sum, the number of non-zero terms is at most $[2l/K]$, hence
there are at most $2^{2l/K}$ possibilities for
the signs of the coordinates of $\Phi(P)$ for any $P\in\cQ_{l,k,i}$.
This finally gives
\[
|\cQ_{l,k,i}|\le \exp[2l/K+2l/K+2l\log(K+1)/K].
\]

Recall that $K=e^{10k}$.
It is an easy calculation to check that
\[
\frac{2\log(K+1)+4}{K}\le\frac{1}{2k}.
\]
Thus
\[
|\cQ_{l,k,i}|\le e^{l/2k}.
\]

Recall that the number of regions is $N\le C\cdot 4^{l/k}$, hence
by Lemma \ref{lemma:decomposition} we have
\[
|\cQ_{l,k}|\le Ce^{l/2k}4^{l/k}\le C\cdot 10^{l/k},
\]
which was to be proved.
\end{proof}


\bibliography{varju}

\bigskip

\noindent{\sc Centre for Mathematical Sciences,
Wilberforce Road, Cambridge CB3 0WA,
England}\\
{\em e-mail address:} pv270@dpmms.cam.ac.uk
\bigskip

\noindent and

\bigskip
\noindent{\sc The Einstein Institute of Mathematics, Edmond J. Safra Campus,
Givat Ram, The Hebrew University of Jerusalem, Jerusalem, 91904, Israel}
\end{document}